\documentclass[11pt]{article}

\usepackage{amsmath}
\usepackage{amssymb}  
\usepackage{amsthm}   
\usepackage{epsfig}

\usepackage{hyperref}

\usepackage[mathcal]{eucal}

\newcommand{\F}{\mathcal{F}}
\newcommand{\J}{\mathcal{J}}

\newcommand{\HT}{\mathcal{H}^\Phi}

\newcommand{\R}{\mathbb{R}}
\newcommand{\BR}{\bar{\mathbb{R}}}
\newcommand{\C}{\mathcal{C}^\Phi}

\newcommand{\x}{r}
\newcommand{\y}{s}

\newcommand{\Phis}{\Phi_{\mathrm{sy}}}

\newcommand{\inner}[2]{\langle{#1},{#2}\rangle}

\newcommand{\tos}{\rightrightarrows}

\newtheorem{theorem}{Theorem}[section]
\newtheorem{lemma}[theorem]{Lemma}
\newtheorem{corollary}[theorem]{Corollary}

\newtheorem{proposition}[theorem]{Proposition}

\newtheorem{definition}[theorem]{Definition}

\title{Fixed Points of Generalized Conjugations}

\author{ M. Marques Alves\thanks{IMPA, Est. D. Castorina 110, 22460-320
    Rio de Janeiro, Brazil 
   ({\tt maicon@impa.br})}\hspace{.5em}\thanks{Partially supported by Brazilian CNPq
scholarship.} 
  \and
  B. F. Svaiter\thanks{ IMPA, Est. D. Castorina 110, 22460-320 Rio de
    Janeiro, Brazil ({\tt benar@impa.br}) }\hspace{.5em}
    \thanks{Partially supported by CNPq
    grants 300755/2005-8, 475647/2006-8 and by PRONEX-Optimization}
}

\date{}

\begin{document}

\maketitle

\begin{abstract}
  Conjugation, or Legendre transformation, is a basic tool in convex
  analysis, rational mechanics, economics and optimization. It maps a function
  on a linear topological space into another one, defined in the
  dual of the linear space by coupling
  these space by meas of the duality product.

  \emph{Generalized conjugation} extends classical conjugation to any pair
  of domains, using an arbitrary coupling function between these
  spaces.
  This generalization of conjugation is now being widely used in
  optima transportation problems, variational analysis and also optimization.

  If the coupled spaces are equal, generalized conjugations 
  define order reversing maps of a   family of functions into itself.
  In this case, is natural to ask for the existence of  fixed
  points of the conjugation, that is, functions which are equal to
  their (generalized) conjugateds.
  Here we prove that any  generalized \emph{symmetric}  conjugation
  has fixed points.
  The basic tool of the proof is a variational principle involving the
  order reversing feature of the conjugation.

  As an application of this abstract result, we will extend to real
  linear topological spaces a fixed-point theorem for Fitzpatrick's
  functions, previously proved in Banach spaces.
  \\
\end{abstract}

\pagestyle{plain}

\section{Introduction}

Fenchel-Legendre conjugation is a basic tool in convex analysis,
classical mechanics and  optimization \cite{RO-Conv,ARN}.
An extension of this conjugation,  proposed by Moreau \cite{moreau1,moreau2}
and known as Generalized Conjugation is
now being used in variational analysis and optimal transportation
\cite{RO-WE,TRU, RuschOpt95, Villani-old,Villani}.
In this work, using a variational principle, we shall prove existence
of fixed points of any generalized (symmetric) conjugation.
This result will be used to extend a fixed-point theorem in the family
of Fitzpatrick's functions, previously proved in a Banach space
setting \cite{SV-Fix03}.

We use the notation $\BR$ for the extended real numbers: 
\[
\BR=\R\cup\{-\infty, \infty\}.
\]
The family of extended real valued functions on a set $E$ will be denoted
by $\BR^E$.
Let $E$ and $F$ be non-empty sets. 
A coupling function 
\begin{equation}
  \label{eq:def.cf}
  \Phi: E \times F \rightarrow \R
\end{equation}
induces two conjugations, $\C _1$ and $\C_2$, defined as follows
\begin{equation}
  \label{eq:def.c1}
  \begin{array}{ll}
\displaystyle     \C_1:\BR^E\to \BR^F,\qquad &\displaystyle 
  \C _1\, h\,(\y)= \sup_{\x \in E} \lbrace \Phi(\x, \y) - h(\x)
  \rbrace\\[1.5em]
\displaystyle   \C_2:\BR^F\to \BR^E,\qquad &\displaystyle 
  \C _2 f(\x)= \sup_{\y \in E} \lbrace \Phi(\x, \y) - f(\y) \rbrace.
  \end{array}
\end{equation}
We refer \cite{RO-WE} to a comprehensive exposition of Generalized
Conjugacy.

Whenever $E=F$ in the coupling function \eqref{eq:def.cf},
both conjugations
(with respect to such coupling function) maps $\BR^E$ into itself.
So, in this case, it does make sense to ask for the existence of
fixed points of these conjugations, that is, $h\in \BR^E$ such that
\[
 \C_1 h=h \mbox{ or } \C_2 h=h.
\]
These fixed points will be called \emph{self-conjugated}
functions with respect to the coupling function $\Phi$. 
 Note that conjugation is order reversing. This feature of conjugation will allow
us to study self-conjugated functions using a variational principle.
This approach has already been used in the context of Fitzpatrick
functions~\cite{SV-Fix03}.

A coupling function $\Phi:E\times E \to\R $ is \emph{symmetric} if
\[
 \Phi(\x,\y)=\Phi(\y,\x), \qquad \forall\, \x,\y\in E.
\]
Note that in the symmetric case, both conjugations in \eqref{eq:def.c1}
coincides, that is, $\C_1=\C_2$.  This additional feature makes the
problem of finding fixed points more manageable.  Surprisingly, symmetry
of the coupling function guarantee existence of self-conjugated
functions.
From now on, conjugation with respect to a symmetric coupling function
$\Phi$ will be denoted by $\C$ ($\C=\C_1=\C_2$). Our aim is to prove
\begin{theorem}[{\sc main result}]
  \label{th:fix}Let $E$ be a non-empty set and $\Phi:E\times E \to \R$
  be symmetric. Take  $g\in \BR^E$.
  \begin{enumerate}
  \item If $\C g \leq g$, then there exists
        $h\in \BR^E$ such that
    \[ \C g\leq \C h=h\leq  g.\]
  \item If $g\in\C(\BR^E)$ and $g\leq \C g $, then there exists 
    $h\in \BR^E$ such that
    \[ g\leq \C h=h\leq \C g.\]
  \end{enumerate}
   In particular, there exists an $h\in\BR ^E$ self-conjugated, that
   is, $h=\C h$.
\end{theorem}

The manuscript is organized as follows: In Section \ref{sec:proof} we
give some basic definitions, prove some technical results and our main
theorem. In Section \ref{sec:a.r} we 
apply the results of Section
~\ref{sec:proof} to the study of non-symmetric conjugations.
In Section \ref{sec:Fitz} we use the main result to extend to linear
topological spaces a fixed point theorem in  Fitzpatrick's family of
functions, previously proved in Banach spaces.

\section{Proof of the main result}
\label{sec:proof}
From now on, $E$ is a  non-empty set and $\Phi$ is a coupling function,
\begin{equation}
  \label{eq:cp.2}
\Phi:E \times E\to\R .  
\end{equation}
Both generalized conjugations (as the classical one) are order
reversing, that is, for any $h,f\in\BR^E$,
\begin{equation}
\label{eq:inv}
 h \leq f \Rightarrow \C_i f \leq  \C_i h,\quad\qquad i=1,2.
\end{equation}
Additionally, for any $h\in \BR^E$,
\begin{equation} \label{eq:J2}
    \C_2\C_1 h\leq h,\qquad
  \C_1\C_2 h\leq h.
\end{equation}
The indicator function of $A\subset E$, is $\delta_A: E \rightarrow \R
\cup \lbrace \infty \rbrace$, 
\begin{equation}
  \label{eq:def.delta}
 \delta_A(\x)=
 \begin{cases}
   0, & \mbox{if } \x\in A,\\
   \infty, & \mbox{otherwise}.
 \end{cases}  
\end{equation}

The following technical result will be needed in the sequel.
\begin{lemma}
  \label{lm:tech1}
  For any $h\in \BR^E$,  $\x_0\in E$ and  $i\in\{1,2\}$
  \begin{eqnarray*}
     \C_i h(\x_0)\leq h(\x_0)&\Rightarrow& \Phi(\x_0,\x_0)/2\leq
     h(\x_0),\\[.4em]
       \C_i h(\x_0)< h(\x_0)&\Rightarrow& \Phi(\x_0,\x_0)/2<
     h(\x_0).
  \end{eqnarray*}
\end{lemma}
\begin{proof}
  If $h(\x_0)=\infty$, then
  trivially $\Phi(\x_0,\x_0)/2<h(\x_0)$. Now, suppose that $ \C_i
  h(\x_0)\leq h(\x_0) < \infty$. Then, by definition \eqref{eq:def.c1}
  \[
   \Phi(\x_0, \x_0) - h(\x_0)\leq \C_i h(\x_0)\leq h(\x_0).
  \]
  Therefore, $ \Phi(\x_0, \x_0)\leq 2 h(\x_0)$. Analogously, if $ \C_i
  h(\x_0)< h(\x_0)$, then the first inequality in the above equation is
  strict and $\Phi(\x_0,\x_0)<2 h(\x_0)$.
\end{proof}

\medskip

To perform our variational analysis, we shall study the family of
functions which are greater than its conjugated.
\begin{definition}
  \label{df:h1}
    \quad$ \HT =  \{ h\in \BR^{E}\;|\; \C_1 h \leq h\}$.
\end{definition}
Latter on we will see that conjugation with respect to the second
variable, $\C_2$, could be used to define the same family.
Fixed points of a generalized (symmetric) conjugation will be obtained
by means of a variational principle, applied on $\HT$.

Note that $\HT$ is non-empty since the function $ h\equiv\infty$
belongs to $\HT$.
Next, we shall prove existence of minimal elements of $\HT$.
  Recall
that if the coupling function \eqref{eq:cp.2} is symmetric, then both
conjugations $\C_1$ and $\C_2$ are identical and we use the notation
$\C=\C_1=\C_2$.
\begin{lemma}
  \label{lm:indutive} Suppose that the coupling function $\Phi:E\times
  E\to \R$ is symmetric.  The family $\HT$ (Def.\ \ref{df:h1}) is
  (downward) \emph{inductively ordered}, i.~e., any totally ordered
  family $\{h_\alpha\}_{\alpha\in \Lambda}\subset \HT$ has a lower
  bound on $\HT$.

  If $g\in\HT$, that is, $ \C g\leq g$, then there exists a minimal $h
  \in \HT$ such that
  \[
  \C g \leq h\leq g.
  \]
  In particular, $\HT$ has minimal elements.
\end{lemma}
\begin{proof}
  Let $\{h_\alpha\}_{\alpha\in \Lambda}$ be a totally ordered
  subset of $\HT$.

  First 
    we claim that
  \[ \C {h_\alpha}\leq
   h_\beta,\qquad \forall\; \alpha,\beta \in \Lambda.
  \]
  To check this claim, take $\lambda,\mu\in \Lambda$ and suppose that
  $h_\lambda\leq h_\mu$.
  Since the conjugation reverse the order, $\C h{_\mu}\leq \C
  {h_\lambda}$. As $\C h_\lambda\leq h_\lambda$ (because
  $h_\lambda\in \HT$), we conclude that
  \[
  \C h_\mu\leq \C h_\lambda \leq {h_\lambda} \leq {h_\mu}.
  \]
  Therefore, $ \C  h_\lambda\leq h_\mu $ and $\C h_\mu \leq
  {h_\lambda}$.  To end the proof of the first claim, use the fact
  that $\{h_\alpha\}_{\alpha\in \Lambda}$ is totally ordered.

  Now define
  \[
  f =\inf_{\alpha\in \Lambda}\;\; {h_\alpha}.
  \]
  Using definition \eqref{eq:def.c1} we get
  \[ \C f =\sup_{\alpha\in \Lambda}\;\; {\C h_\alpha},
  \]
  which, combined with the previous claim and the definition of $f$ yields
  \[
    \C f\leq f.
  \] 
  So, $f \in \HT$ and is a lower bound for the family
  $\{h_\alpha\}_{\alpha\in \Lambda}$.

  To prove the second part of the lemma, use Zorn's Lemma (see
  \cite[Theorem 2, pp 154 and Corollary 1, pp 155]{bourbaki}) to
  conclude that for any $g\in \HT$ there exists a minimal $h\in \HT$
  such that $h\leq g$. Applying $\C$ in this inequality we obtain $\C
  g\leq \C h\leq h$, where the second inequality comes from the
  inclusion $h\in\HT$. To end the proof, note that $\HT$ is non-empty.
\end{proof}

\begin{lemma}
  \label{lm:minimal}  Suppose that the coupling function  $\Phi:E\times E\to \R$
  is symmetric.
  If $h=\C h$ then $h$ is a minimal element of $\HT$.
\end{lemma}
\begin{proof}
  Suppose that $g\in \HT$ and $g \leq h$.
  Applying $\C$  on this inequality gives $  \C h\leq  \C g$.
  Therefore,
  \[
  h=\C h \leq \C g \leq g
  \]
  where the last inequality follows from the assumption $g\in\HT$.
  Altogether we have $g\leq h$ and $h\leq g$.  So, $g=h$ and $h$ is
  minimal in $\HT$.
\end{proof}

\bigskip

To prove Theorem \ref{th:fix} now, it is sufficient to prove the converse of
Lemma \ref{lm:minimal}.

\begin{lemma}
  \label{lm:main}
  Suppose that the coupling function  $\Phi:E\times E\to \R$
  is symmetric. Then
   $h\in \BR^E$ is a minimal element of $\HT$ if and only if $h=\C h$.
\end{lemma}
\begin{proof}
  We already know, by Lemma \ref{lm:minimal}, that if  $h=\C h$ then
  $h$ is minimal in $\HT$.

  Suppose now that $h$ is minimal in $\HT$.  We shall prove that
  \begin{equation}
    \label{eq:abs.1}
  \C h(\x_0)< h(\x_0)    
  \end{equation}
  cannot hold.  If this  inequality holds, then
  by  Lemma \ref{lm:tech1}, $\Phi(\x_0,\x_0)/2< h(\x_0)$. Hence there exists
  $t_0\in \R$ such that
  \begin{equation}
    \label{eq:abs.2}
     \max\left\{ \C h(\x_0), \Phi(\x_0,\x_0)/2\right\}\;\leq t_0 < h(\x_0).
  \end{equation}
  Define
  \begin{equation}
    \label{eq:abs.3}
    g=\min \{h,\delta_{\x_0}+t_0\}.   
  \end{equation}
  We will prove that $g\in \HT$, and this will lead to a contradiction.
  Using \eqref{eq:def.c1}, 
we get
  \begin{eqnarray*}
    \C g (\x)&=&\max \{\C h(\x),\C (\delta_{\x_0}+t_0)(\x)\}\\
    &=&\max \{\C h (\x),\Phi(\x,\x_0)-t_0\}.
  \end{eqnarray*}
  For any $\x\in E$,
  \[
  \Phi(\x,\x_0)-t_0 \leq  \Phi(\x,\x_0)-\C h (\x_0)
  \leq  (\C)^2 h (r) \leq   h(r),
  \]
   and  $\C h\,(r)\leq h(r)$.
  Hence,
  \[
   \C g\leq h.
  \]
  As $\Phi(\x_0,\x_0)-t_0\leq t_0$
  and $\C h (\x_0)\leq t_0$, we also conclude that
  \[
  \C g \leq \delta_{\x_0}+t_0  .
  \]
  Combining the two above inequalities with \eqref{eq:abs.3}
  we obtain $\C g\leq g$. Therefore,
  \[
  g\in \HT.
  \]  
  As $g\leq h$ and $h$ is minimal in $\HT$,  $g=h$ and, 
  in particular,
  \[
  h(\x_0)=g(\x_0).
  \]
  From the definition of $g$ we have $g(\x_0)=t_0<h(\x_0)$, which is a
  contradiction. So, \eqref{eq:abs.1} can not hold in any $r\in
  E$. As $\C h\leq h$, we conclude that $\C h=h$.
\end{proof}

\begin{proof}[Proof of Theorem \ref{th:fix}]
Combining Lemma \ref{lm:indutive} with Lemma \ref{lm:main} we conclude
that item 1 holds and that there exists a self-conjugated function
$h=\C h$.

To prove item 2, assume that $g=\C g_0$ and $g\leq \C g$. Applying $\C$
on this inequality we obtain $(\C)^{2} g\leq \C g$, which is equivalent to
\[ 
\C (\C)^2\, g_0\leq (\C)^2\, g_0.
\] 
Applying item 1 to $(\C)^2 g_0$ we conclude that there exists $h$,
\[
(\C)^3 g_0\leq
h=\C h\leq (\C)^2\, g_0.
\] 
Note that $(\C)^3 g_0=(\C)^2\C\, g_0\leq \C\,g_0$. 
Applying $\C$ to the inequality $(\C) ^2 g_0\leq g_0$ we also have
$\C g_0\leq (\C)^3 g_0$. Hence\footnote{In fact, $(\C)^3=\C$, which
  is a property of any symmetric conjugation.},
 $(\C)^3 g_0=\C g_0$, which combined
with the above equation yields
\[ g=\C g_0\leq h=\C h\leq (\C)^2\, g_0=\C g\,.\tag*{\qedhere} \]
\end{proof}

\section{Additional results}
\label{sec:a.r}
Here we present some additional results to Section~\ref{sec:proof}
which were not necessary for proving the main theorem.
Non-symmetric conjugation will also be discussed with more details.

\begin{proposition}
  \label{pr:h1.h2}
  For any $h\in \BR^E$,
  the following conditions are equivalent
  \begin{enumerate}
  \item  $ \C_1 h\leq h$,
  \item $\C_2 h\leq h$,
  \item $\max\{ \C_1 h, \C_2 h\}\leq h$.
  \end{enumerate}
\end{proposition}
\begin{proof}
  Suppose that 1 holds, $ \C_1 h\leq h$.
  As $\C_2$ is order reversing,  applying $\C_2$ on this inequality 
  we get
  \[\C_2 h\leq   \C_2\C_1 h,
  \]
  which, combined with the first inequality in \eqref{eq:J2} yields
  $ \C_2 h\leq h$. So condition 1 implies condition 2.
  To prove that condition 2 implies 1 apply $\C_1$ on both sides of
  the inequality $\C_2 h\leq h$ and follows the same reasoning.

  Condition 1 or 2, being equivalent,  implies condition 3, which is
  equivalent  to condition 1 and 2.
\end{proof}

We define the {\it symmetrization} of $\Phi$ as $\Phis $,
\begin{equation}
  \label{eq:def.sym}
  \Phis :E\times E\to \R,\qquad \Phis(\x,\y)=\max\{
\Phi(\x,\y), \Phi(\y,\x)\}.
\end{equation}
Notice that $\Phis $ is symmetric. Direct
calculation gives
\begin{equation}
 \label{eq:c.til} \mathcal{C}^{\Phis }h=\max\{ \C_1 h, \C_2 h\}\,,
\end{equation}
which, combined with  Definition \ref{df:h1} yields
\[
\mathcal{H}^{\Phis } = \{ h\in \BR^{E}\;|\;
\mathcal{C}^{\Phis }h \leq h \}
=\{ h\in \BR^{E}\;|\;\max\{ \C_1 h, \C_2 h\} \leq h \}
.
\]
Using Proposition \ref{pr:h1.h2} we obtain alternative
characterizations of $\HT$: 
\begin{equation}
  \label{eq:def.h}
\begin{array}{rcl}
  \HT &= & \{ h\in \BR^{E}\;|\; \C_1 h \leq h\}\\
  &=& \{ h\in \BR^{E}\;|\; \C_2 h \leq h\}\\
  &=& \{ h\in \BR^{E}\;|\;\max\{ \C_1 h, \C_2 h\} \leq h \}=
  \mathcal{H}^{\Phis }\;.
\end{array}  
\end{equation}
With 
the above equation, now it is straightforward to generalize
Lemma~\ref{lm:indutive} and Lemma~\ref{lm:main} to non-symmetric conjugations.
\begin{proposition}
  \label{pr:gen}
  Let $\Phi:E\times E\to \R$ be a generic coupling function. Then
  \begin{enumerate}
  \item The family $\HT$ is (downward) inductively ordered.
  \item For any $g\in \HT$ there exists a minimal $h\in \HT$,  such that,
    \[ \max\{ \C_1 g, \C_2 g\}\leq h\leq g.\]
  \item The family $\HT$ has minimal elements
  \item $h\in \HT$ is minimal if and only if\;
    $ \mathcal{C}^{\Phis }h=\max\{ \C_1 h, \C_2 h\} = h$.
  \end{enumerate}
\end{proposition}

Also in the non-symmetric case, 
fixed points of the conjugations $\C_1$ or $\C_2$
are minimal elements of $\mathcal{H}^{\Phis }$.
\begin{proposition}
  \label{pr:fix.ns}
  If $h=\C_1 h$ or $h=\C_2 h$, then $h\in\mathcal{H}^{\Phi}$ and is
  minimal. 
\end{proposition}
\begin{proof}
  If $h=\C_1 h$ then, in particular $\C_1 h\leq h$. Hence by
  \eqref{eq:def.h} $h\in \mathcal{H}^{\Phis }$ and so
  \[\C_2 h\leq  h= \C_1 h  \]
  which implies $\max\{ \C_1 h, \C_2 h\}=h$ so that by
  \eqref{eq:c.til} $\mathcal{C}^{\Phis }h=h$.  Now apply
  Lemma \ref{lm:minimal} to conclude that $h$ is minimal in
  $\mathcal{H}^{\Phis }$.
  The case $\C_2 h=h$ follows the same proof, interchanging $\C_1$ and
  $\C_2$.
\end{proof}

A natural question is whether Lemma~\ref{lm:main} can be extended to a
non-symmetric $\Phi$. The answer is negative, as exposed in the next
example.

\noindent
Take  $E=\{a,b\}$ with $a\neq b$ and $\Phi:E\times E\to \R$
\[
\begin{array}{ll}
  \Phi(a,a)=0, &\Phi(a,b)=-3, \\[.5em]
  \Phi(b,a)=0,  &\Phi(b,b)=-3. 
\end{array}
\]
For the function $h:E\to\BR$, $h(a)=1$ and $h(b)=-1$,
we have
\[
\begin{array}{ll}
  \C_1 h(a)=1, & \C_1 h(b)=-2, \\[.5em]
  \C_2 h(a)=-1,  &\C_2 h(b)=-1. 
\end{array}
\]
As $h=\max\{ \C_1 h,\C_2 h\}$, by \eqref{eq:c.til} and
Lemma~\ref{lm:minimal}, $h$ is minimal in $\mathcal{H}^{\Phis}$ but is not a fixed point of $\C_1$ or $\C_2$.

Lemma~\ref{lm:tech1} applied to family $\mathcal{H}^{\Phi}$ yields the
following result, which relates these functions $h\in
\mathcal{H}^{\Phi}$ with the coupling function $\Phi$ and the
generalized subdifferential.
\begin{corollary}
  \label{cr:lm.tech1}
  For any $h\in \mathcal{H}^{\Phi}$:
  \begin{enumerate}
  \item  $ \Phi(r,r)/2\leq h(r)$ for all $r\in E$.
  \item If  $ h(r_0)=\Phi(r_0,r_0)/2$, then
  \[ \C_1 h(r_0)=\C_2 h(r_0)=\Phi(r_0,r_0)/2\]
  and $\x_0\in\partial^\Phi _1 h(\x_0)$, $r_0\in\partial^\Phi _2 h(\x_0)$,
 that is, 
 for all $r\in E$
  \[ h(r_0)+\big[\Phi(r,r_0)-\Phi(r_0,r_0)\big]\leq h(r),\]
 \[ h(r_0)+\big[\Phi(r_0,r)-\Phi(r_0,r_0)\big]\leq h(r).\]
  \end{enumerate}
\end{corollary}
\begin{proof}
  Item 1 follows directly from  Definition \ref{df:h1} and  the first
  implication on Lemma \ref{lm:tech1}.

  To prove item 2, first use the second implication on  Lemma
  \ref{lm:tech1} to conclude that $\C_i h(\x_0)\geq
  h(\x_0)$. Now use \eqref{eq:def.h} to conclude that this inequality
  holds as an equality.  As $\C_1 h(\x_0)=h(\x_0)=\Phi(\x_0,\x_0)/2$,
  by \eqref{eq:def.c1}
  \[
  h(\x_0)\geq  \Phi(\x,\x_0)-h(\x) 
  \]
  for all $\x\in E$. Hence
  \begin{eqnarray*}
    h(\x)&\geq&\Phi(\x,\x_0)-h(\x_0)\\
    &=&\Phi(\x,\x_0)-2 h(\x_0) +h(\x_0)=\Phi(\x,\x_0)-\Phi(\x_0,\x_0) +h(\x_0).
  \end{eqnarray*}
  The last inequality follows from the same arguments.
\end{proof}

\section{Self-conjugated Fitzpatrick functions, or fixed points of the
  $\J$ mapping}
\label{sec:Fitz}
Now we will use Theorem~\ref{th:fix} to study self-conjugated
Fitzpatrick's functions. 

In this section  $X$
is a real linear  topological space and $X^*$  its dual, endowed with
the  weak-$*$ topology.
In $X\times X^*$, consider the canonical product topology.
 Use the
notation $\langle x,x^*\rangle$ for the duality product
\[ \langle x,x^*\rangle=x^*(x), \qquad x\in X,\; x^*\in X^*.\]
A point to set operator $T:X\tos X^*$ is a relation on $X$ to $X^*$:
\[ T\subset X\times X^*\]
and $x^*\in T(x)$ means $(x,x^*)\in T$.
An operator $T:X\tos X^*$ is
 \emph{monotone} if
\[
\langle x-y,x^*-y^*\rangle\geq 0,\qquad \forall (x,x^*),(y,y^*) \in T.
\]
The operator $T$ is \emph{maximal monotone} if it is monotone and
maximal in the family of monotone operators of $X$ into $X^*$ (with
respect to order of inclusion).

Fitzpatrick proved  that associated to any maximal
monotone operator in $X$ there exists a family of lower semicontinuous
convex functions in $X\times X^*$ which characterize the operator:
\begin{theorem}[\mbox{ \cite[Theorem 3.10]{FITZ-Rep88}}]
  \label{th:fitz}
  If $T$ is a maximal monotone operator on a real linear topological
  space $X$, then
  $\varphi_T:X\times X^*\to \BR$
  \begin{equation}
    \label{eq:Fitzfunc}
    \varphi_{T}(x, x^*) = \sup_{(y,
      y^*) \in T} \langle x - y, y^* - x^* \rangle + \langle x, x^*
    \rangle
  \end{equation}  
  is the smallest element of the family $\F_T$,
  \begin{equation}
    \label{eq:def.ft}
    \F_T=\left\{ h\in \BR^{X\times X^*}
      \left|
      \begin{array}{ll}
        h\mbox{ is convex and lower semicontinuous}\\
        \inner{x}{x^*}\leq h(x,x^*),\quad \forall (x,x^*)\in X\times X^*\\
        (x,x^*)\in T 
        \Rightarrow 
        h(x,x^*) = \inner{x}{x^*}
       \end{array}
       \right.
      \right\}
    \end{equation}
    Moreover, for any $h\in \F_T$,
    \[
    (x,x^*)\in T\iff    h(x,x^*)=\inner{x}{x^*}.
    \]
\end{theorem}
Note that \emph{any} $h\in \F_T$ fully characterizes $T$.
Fitzpatrick's family of convex representation of maximal monotone
operators was recently rediscovered~\cite{BU-SVset02,MarLeg-The}
and since then, this subject has been object of intense research.  
Note that the family $\F_T$ is closed under the $\sup$ operation. Therefore
\begin{proposition}
  \label{pr:bigest}
  Let  $T$ be a maximal
  monotone operator on a real linear topological space $X$.
  There exists a (unique) maximum element $\sigma_T\in\F_T$,
  \[\sigma_T=\sup_{h\in\F_T}\;\{h\}.
  \]
\end{proposition}
\noindent
The maximal representation $\sigma_T$ and the structure of its
epigraph were studied on a Banach space setting
in~\cite{BU-SVset02,Svaiter-FE}.

Fenchel-Legendre conjugate of $f:X\to \BR$ is defined as $f^*:X^*\to\BR$,
\begin{equation}
  \label{eq:def.conj.classic}
  f^*(x^*)=\sup_{x\in X} \inner{x}{x^*}-f(x). 
\end{equation}
Define, as in \cite{BU-SVset02}
\begin{equation}
  \label{eq:def.j}
   \J: \BR^{X\times X^*}\to \BR^{X\times X^*},\quad (\J h)(x,x^*)=h^*(x^*,x).
\end{equation}
Hence, for all $(x,x^*)\in X\times X^*$, 
\begin{equation}
  \label{eq:j.eq}
  \begin{array}{rcl}
 \J h\,\,(x,x^*)&=&\sup_{(y,y^*)\in X\times X^*}
 \big\langle(y,y^*)\,,\;(x^*,x)\big\rangle-h(y,y^*)\\
 &=&\sup_{(y,y^*)\in X\times X^*}
 \langle x, y^* \rangle + \langle y, x^* \rangle-h(y,y^*).    
  \end{array}
\end{equation}
Direct use of \eqref{eq:def.j} or \eqref{eq:j.eq} and
\eqref{eq:def.conj.classic} shows that
\begin{equation}
  \label{eq:j2}
  \J^2 f=f^{**}\leq  f,\qquad \forall f\in \BR^{X\times X^*}.
\end{equation}

The family $\F_T$ is invariant under $\J$ in a Banach space
setting~\cite{BU-SVset02}.  Here we extend this result to linear
topological spaces.
Note that 
  if $X$ is not Hausdorff, any lower semicontinuous function must
  assume only one value at each family of non-separable points. So,
  in dealing with lower semicontinuous functions,
  whenever we need $X$ to be Hausdorff, we
  can work in $X/(X^*)^\dag$, where $(X^*)^\dag$ is the annihilator of $X^*$. 
\begin{theorem} \label{teo:invari}
 Let  $T$ be a maximal
 monotone operator on a real linear topological space $X$.
 The application $\J$ maps $\F_T$ into itself.
 If $X$ is 
 locally convex, $\J$ maps  $\F_T$ \emph{onto} itself. 
\end{theorem}
\begin{proof}
Define $\pi:X\times X^*\to \R$, 
\[
\pi(x,x^*)=\inner{x}{x^*}.
\]
Take $h\in \F_T$. 
By Theorem~\ref{th:fitz} 
\[
\varphi_T\leq  h \leq \delta_{T}+\pi.
\]
As $\J$ is order
reversing, applying this mapping on both terms of this inequalities
we obtain
\[\J (\delta_T+\pi)\leq \J h\leq \J \varphi_T.\]
Direct use of \eqref{eq:j.eq} and \eqref{eq:Fitzfunc} yields $\J
(\delta_T+\pi)=\varphi_T$, which applied to the above inequality
yields
\[ \varphi_T\leq \J h \leq \J^2(\delta_T+\pi).\]
Again by Theorem~\ref{th:fitz} $\pi\leq \varphi_T$. Combining this
result with the above inequalities and \eqref{eq:j2} we obtain
\[\pi\leq \J h\leq \delta_T+\pi.\]
According to the above equation, $\inner{x}{x^*}\leq \J h (x,x^*)$ for
all $(x,x^*)\in X\times X^*$, with equality if $(x,x^*) \in T$.
By definition \eqref{eq:def.j} or \eqref{eq:j.eq}, $\J h $ is convex
and lower semicontinuous. Therefore, $\J h\in \F_T$.

Assume now that  $X$ is locally convex. 
Take $h\in\F_T$. As $h$ is
convex and lower semicontinuous,
\[ \J (\J h)=h^{**}=h.\]
As $\J h\in \F_T$, we
obtain  $h=\J^2 h\in \J (\F_T)$.
\end{proof}

Now we are ready to extend 
the fixed point of
theorem~\cite{SV-Fix03} to linear topological spaces.
\begin{theorem}
  \label{th:main.1}
  Let $T$ be a maximal monotone operator on a real linear  topological
  space $X$.
  \begin{enumerate}
  \item If $g\in\F_T$ and $ \J g\leq g$
    then there exists $h\in\F_T$ such that
    \[
    \J g\leq \J h=h\leq g.
    \]
  \item  If $g\in \J(\F_T)$ and $ g\leq \J g$
    then there exists $h\in\F_T$ such that
    \[
    g\leq h=\J h\leq \J g.
    \]
   \end{enumerate}
  In particular, there exists $h\in \F_T$ such that $h=\J h$.
\end{theorem}
\begin{proof}
Take $E:= X \times X^{*}$ and consider the coupling function 
$\Phi$,
\[
\Phi: (X\times X^*)\times (X\times X^*)\to\R, \quad
\Phi\big(\;(x, x^*)\,, \;(y, y^*)\;\big):=
   \langle x, y^* \rangle + \langle y, x^* \rangle\;. 
\]
Note that $\Phi$ is symmetric. Moreover, using
\eqref{eq:def.conj.classic}, 
\eqref{eq:def.j} and \eqref{eq:def.c1} we have
\[ \C=\J.
\]

If $g\in \F_T$ and $\J g\leq g$, this means $\C g\leq g$. Using item 1
of Theorem~\ref{th:fix} we conclude that there exists $h\in \HT$ such
that
\[ \J g\leq \J h=h\leq g.\] Now we must show that $h\in \F_T$.
As $\J h=h$ is the supermom of a family of continuous affine
functionals on $X\times X^*$, we conclude that $h$ is convex and lower
semicontinuous.
Since $\J g\in \F_T$,
for any $(x,x^*) \in X\times X^*$,
\[ \inner{x}{x^*}\leq \J g(x, x^*)\leq h(x,x^*)\leq g(x,x^*).\]
In particular, $\inner{x}{x^*}\leq h(x,x^*)$. If
$(x,x^*)\in T$ then, as $g\in \F_T$, $g(x,x^*)=\inner{x}{x^*}$,
the above inequalities hold as equalities and
$h(x,x^*)=\inner{x}{x^*}$.
Therefore, $h\in \F_T$ and item 1 holds.

To prove item 2 use item 2 of   Theorem~\ref{th:fix}  and repeat the
reasoning used to prove item 1.

To end the proof, we must show that there exists a fixed point of $\J
$ in $\F_T$. As $\sigma_T$ is maximal in $\F_T$ and $\J
\sigma_T\in\F_T$, we conclude that $\J \sigma_T\leq \sigma_T$. Now,
apply item 1 of the theorem.
\end{proof}

\begin{corollary}
  Let $T$ be a maximal monotone operator on a real linear locally convex topological
  space $X$.  If $g \in \F_T$ and $g\leq \J g$
  then there exists $h\in\F_T$ such that
  \[
   g\leq \J h=h\leq \J g.
  \]
\end{corollary}
\begin{proof}
  First use Theorem~\ref{teo:invari} to conclude that $g\in \J(\F_T)$
  and then apply item  2 of Theorem~\ref{th:fix}.
\end{proof}
\vspace{0.7 cm}



\def\cprime{$'$}

\end{document}